\documentclass[12pt,reqno]{amsart}
\title{A Spectral Gap Absorption Principle}
\usepackage{amsmath, amsthm, amsopn, amssymb, microtype}
\usepackage{mathrsfs} 
\usepackage[all,cmtip]{xy}
\usepackage{enumerate, tikz, etoolbox, intcalc, geometry, caption, subcaption, afterpage, dynkin-diagrams}
\usepackage[english]{babel}
\usepackage{pdfpages}

\author{Yuval Gorfine}
\address{Faculty of Mathematics and Computer Science, The Weizmann Institute of Science, 234 Herzl Street, Rehovot 7610001, ISRAEL}
\email{Yuval.Gorfine@gmail.com}

\usepackage[colorlinks=true,urlcolor=blue,pdfborder={0 0 0}]{hyperref}
\hypersetup{linkcolor=[rgb]{0,0,0.6}}
\hypersetup{citecolor=[rgb]{0,0.6,0}}
\usepackage[nameinlink]{cleveref}
\crefname{theorem}{Theorem}{Theorems}
\crefname{theorem}{Theorem}{Theorems}
\crefname{mainthm}{Theorem}{Theorems}
\crefname{lemma}{Lemma}{Lemmas}
\crefname{lem}{Lemma}{Lemmas}
\crefname{remark}{Remark}{Remarks}
\crefname{prop}{Proposition}{Propositions}
\crefname{defn}{Definition}{Definitions}
\crefname{corollary}{Corollary}{Corollaries}
\crefname{cor}{Corollary}{Corollaries}
\crefname{section}{Section}{Sections}
\crefname{figure}{Figure}{Figures}
\crefname{quest}{Question}{Questions}

\newtheorem{theorem}{Theorem}[section]
\newtheorem{lemma}[theorem]{Lemma}
\newtheorem{proposition}[theorem]{Proposition}
\newtheorem{corollary}[theorem]{Corollary}

\theoremstyle{definition}
\newtheorem{definition}[theorem]{Definition}

\newtheorem{remark}[theorem]{Remark}
\newtheorem*{conjecture}{Conjecture}

\newcommand{\abs}[1]{\left| #1 \right|}

\begin{document}

\begin{abstract}
    We show that unitary representations of simply connected, semisimple algebraic groups over local fields of characteristic zero obey a spectral gap absorption principle: that is, that spectral gap is preserved under tensor products. We do this by proving that the unitary dual of simple algebraic groups is filtered by the integrability parameter of matrix coefficients. This is a filtration of closed ideals that captures every closed subset of the dual that doesn't contain the trivial representation. In other words, we show that a representation has a spectral gap if and only if there exists some $p < \infty$ such that its matrix coefficients are in $L^{p+\epsilon}(G)$ for every $\epsilon>0$. Doing this, we continue the work of Bader and Sauer in this area and prove a conjecture they phrased. We also use this principle to give an affirmative solution to a conjecture raised by Bekka and Valette: the image of the restriction map from a semisimple group to a lattice is never dense in Fell topology. 
\end{abstract}
\maketitle
\section{Introduction}
Let $G$ be a locally compact group, and consider its unitary representation theory. A natural question that may be asked in this area is the question of absorption, under tensor products, of certain properties of representations: given a unitary representation $\pi$ of $G$ that possess a certain property, does the representation $\pi \otimes \rho$ possess this property, for every unitary representation $\rho$? (here and everywhere, $\otimes$ denotes the Hilbertian tensor product). This is a fair question, as tensor products occur naturally, and we want nice properties of representations to prevail.

If $G$ is not amenable, harmonic analysis on $G$ takes a turn from classical harmonic analysis, as it is no longer true that every representation of $G$ is weakly contained in the left regular representation, $\lambda$. The representations that are weakly contained in the left regular representation are called \textit{tempered} representations. The distinction between tempered and non-tempered representations is essential, and it is useful to know that tempered representations obey an absorption principle: it is an easy, classical result that temperedness is preserved under tensor products (this is an immediate corollary from Fell's absorption principle).

We now restrict our attention to the case of semisimple groups, and consider the notion of \textit{spectral gap}. More precisely, let $G_i$ be the group of $k_i$-points of an almost simple, connected and simply connected algebraic $k_i$-group that is $k_i$-isotropic, where $k_i$ is some local field of characteristic zero. Consider the group $G=\prod_{i=1}^n G_i$. A representation is said to have a \textit{spectral gap} if it doesn't admit almost invariant vectors, that is, if its support doesn't contain the trivial representation. In a way, tempered representations are exactly the ones having a "maximal spectral gap". Hence we consider the following theorem, which is our main result, as a generalization of the absorption of temperedness:
\begin{theorem}[Theorem \ref{main theorem}]\label{main_intro}
    Let $\pi$ and $\rho$ be two unitary representations of $G$, such that $1 \nprec \pi$. Then $1 \nprec \pi \otimes \rho$. That is, spectral gap is preserved under tensor products.
\end{theorem}
 
 To motivate the proof of the theorem, we will illustrate here a proof of the temperedness absorption theorem, that only applies to semisimple Lie groups and semisimple algebraic groups over local fields (which are our object of interest in this paper), even though this is an easy theorem. For semisimple groups, the set of irreducible tempered representations, the \textit{tempered dual}, was characterized by Harish-Chandra through the integrability parameter of matrix coefficients: an irreducible representation is in the tempered dual if and only if its matrix coefficient functions are in $L^{2+\epsilon}(G)$ for every $\epsilon>0$. This description of the tempered dual shows that it is a \textit{closed ideal}, that is, that the support of the tensor product of a representation from this set with any other representation is supported on this set as well. The way we wish to generalize this idea to other absorption principles is the following: given a property for which we want to prove an absorption principle, we will try to find many closed ideals in $\hat{G}$, each of them consisting of representations that admit this property, such that any closed subset of representations that admit this property will end up being a subset of one of these ideals. For tempered representations we found one closed ideal that captures the entire phenomenon; for spectral gap we will need to use an infinite family of closed ideals. To this end, we introduce the closed ideals $\hat{G}_p$:
\begin{definition}[Definition \ref{p_ideals}]
    The set $\hat{G}_p \subset \hat{G}$ is the set of all irreducible representations $\pi \in \hat{G}$ for which there exists a dense subspace $H_0 \subset H_\pi$ such that for every $u,v\in H_0$, the matrix coefficient function $g \rightarrow \langle\pi(g) u,v \rangle$  is in $L^{p+\epsilon}(G)$ for every $\epsilon>0$.
\end{definition}
Unlike $\hat{G}_2$, the set of tempered representations, the fact that $\hat{G}_p$ is a closed subset (and in fact a closed ideal) for semisimple algebraic groups is more difficult. Following the work of Samei and Wiersma \cite{Samei-Wiersma} on exotic $C^*$-algebras, it was proved in \cite{deLaat-Siebenand} for groups admitting the Kunze-Stein property. Simple algebraic groups over local fields are Kunze-Stein, and hence the set of subsets of the form $\hat{G}_p$ is a potential filtration of closed ideals for the unitary dual, that can prove the spectral gap absorption principle. For the filtration to work, we need to restrict our attention to the case of simple algebraic groups rather than semisimple. Bader and Sauer used a filtration of the unitary dual by the ideals $\hat{G}_p$ to prove theorem \ref{main_intro} in some cases \cite[Section~4.2]{Bader-Sauer}. They proved the theorem for quasi-split p-adic groups, and explained that the higher rank case and the archimedean case were already handled by previous authors. The full p-adic rank one case was left in their paper as a conjecture \cite[Conjecture~4.6]{Bader-Sauer}. We give in this paper a positive answer to their conjecture, hence completing the proof for all almost simple, simply connected algebraic groups over local fields. That is, we prove that the ideals $\hat{G}_p$ indeed filtrate the closed subsets of $\hat{G}$ that don't contain the trivial representation.
\begin{theorem}[Theorem \ref{filtration}]\label{filtrationIntro}
    Let $G=\textbf{G}(k)$ be the group of $k$-points of an almost simple, connected and simply connected algebraic $k$-group $\textbf{G}$ that is $k$-isotropic, where $k$ is some local field of characteristic zero. Let $A$ be a closed subset of $\hat{G}$ such that $1 \notin A$. Then, there exists a $p \in [2,\infty)$ such that $A \subset \hat{G}_p$. 
\end{theorem}
The analogy between $\hat{G}_p$ and the size of the spectral gap comes from the simplest case, that of $SL_2(\mathbb{R})$. $SL_2(\mathbb{R})$ has only one complementary series, parameterized by the segment $[0,1]$, where in $0$ we have a tempered representation from the principal series and in $1$ lies the trivial representation (as well as two discrete series representations). Here, as the parameter $s$ runs from $0$ to $1$, the representations (that converge to the trivial representation) escape every ideal of the form $\hat{G}_p$. In other words, $\frac{2}{p}$ may be suggested as a measure for the size of the spectral gap of a representation supported in $\hat{G}_p$, which gives sense in this specific case to the idea that "tensor products can't decrease the spectral gap".

To prove theorem \ref{filtrationIntro} we employ the Langlands classification theorem, and use both classical and recent results to control the integrability parameter of the matrix coefficients of representations coming from different sets of complementary series. We deduce that a sequence of representations that escapes every ideal of the form $\hat{G}_p$ already converges to the trivial representation. 

As an application of the spectral gap absorption principle, we give an affirmative solution to a conjecture raised by Bekka and Valette in \cite{BeV}. They considered the restriction map from $\widetilde{G}$ to $\widetilde{\Gamma}$, where $G$ is a connected semisimple Lie group, $\Gamma$ is a lattice and $\widetilde{H}$ denotes the separable dual of a locally compact, second countable group $H$ (that is, the set of unitary representations on separable Hilbert spaces up to unitary equivalence, endowed with the Fell topology). Bekka and Valette suggested that the image of this map is never dense in Fell topology, namely that there exist unitary representations of $\Gamma$ that are not weakly contained in any representation restricted from $G$. They proved this conjecture in many cases. We complete the proof for all semisimple groups, including semisimple groups over non-archimedean local fields of characteristic zero. More precisely, we show that non-trivial finite dimensional representations are never weakly contained in a representation restricted from the group. We do this using the spectral gap absorption principle, together with the fact that the $G$-representation $L^2_0(G/\Gamma)$ has a spectral gap. 
\begin{theorem}[Theorem \ref{BekVal}]\label{BekValIntro}
    Let $G$ a group that belongs to one of the following families:
    \begin{enumerate}
        \item Connected semisimple Lie groups with finite center.
        \item $G=\prod_{i=1}^n \textbf{G}_i(k_i)$, where each $k_i$ is a local field of characteristic zero and $\textbf{G}_i$ is an almost simple, connected and simply connected algebraic $k_i$-group, which is $k_i$-isotropic.
    \end{enumerate}
    Let $\Gamma$ be a lattice in $G$. Then, the image of the restriction map from $\tilde{G}$ to $\tilde{\Gamma}$ has a non-dense image in Fell topology. Moreover, every non-trivial finite dimensional representation of $\Gamma$ is not weakly contained in any representation restricted from $G$. 
\end{theorem}

\subsection{Structure of the Paper}
In section \ref{sec2} we establish the basic notions and terminology that we need, most of which are representation theoretic. We first discuss unitary representations of locally compact groups, then restrict our attention to tdlc groups and discuss smooth and admissible representations, and then further restrict our attention to reductive p-adic groups. In section \ref{sec3} we explain how to deduce theorem \ref{main_intro} from theorem \ref{filtrationIntro}. Section \ref{sec4} is then devoted to prove theorem \ref{filtrationIntro}. In section \ref{sec5} we use the spectral gap absorption principle to prove theorem \ref{BekValIntro}. In the appendix, we give the classification of rank-1 p-adic groups, due to Tits. The non-quasi-split groups in this list are the ones for which our results \ref{main_intro} and \ref{filtrationIntro} are new.

\subsection{Acknowledgments}
I would like to express my gratitude to many people who helped with this project. First of all, I want to thank my advisor Uri Bader for introducing me to this problem and to this beautiful area of mathematics, and for helping me a lot throughout the way. I am grateful for the useful advice that I received from numerous people (some of them may not even know that they have given me useful advice - but they have): Erez Lapid, Dmitry Gourevitch, Alex Lubotzky, Yair Glasner, Roman Sauer and Bachir Bekka. I would also like to thank Marko Tadic, Gordan Savin and Michael Cowling for communicating critical advice to my advisor Uri, that was later passed on to me. Last but not least I want to thank Alon Dogon, Michael Glasner and Guy Kapon, who helped and supported me and with whom I have had many discussions regarding this project.

\section{Preliminaries}\label{sec2}
In the preliminaries, we recall the basics of representation theory of reductive p-adic groups. We distinguish  two different categories: unitary representations and smooth, admissible representations. We start by fixing some terminologies about the unitary dual of a locally compact group. We then recall the basics of (smooth) representation theory of tdlc groups, and explain the relation between the two theories (admissible and unitary representations), which applies in the reductive case due to a result of Bernstein. After that, we recall the machinery of parabolic induction and quote the Langlands classification theorem.

\subsection{The Unitary Dual}
For an introduction to the unitary dual, we direct the reader to \cite[Chapter~7]{Folland}, to \cite[Part~II]{BekHarVal}, or to \cite{Dixmier}, which are our main references here. 

Let $G$ be a locally compact group. We denote by $\hat{G}$  the \textit{unitary dual} of $G$, that is, the set of (equivalence classes of) irreducible, unitary representations of $G$. A unitary representation is assumed to be on a Hilbert space (and not just pre-Hilbert) and strongly continuous. Irreducibility is in the topological sense - no closed invariant subspace exists. By the \textit{trivial representation} of $G$ we refer to the one dimensional representation by the identity operator, and sometimes denote it by $1$. Throughout this section, all representations are assumed to be unitary.

One endows the unitary dual, as a set of points, with several structures. We will mostly be concerned with the standard way of defining a topology on this set, which is called the \textit{Fell Topology}. 
\begin{definition}[Weak Containment]
    Let $S$ be a set of representations of $G$, and $\pi$ another representation. Then $\pi$ is said to be \textbf{weakly contained} in $S$, denoted by $\pi \prec S$, if for all $f\in L^1(G)$, one has the following inequality in operator norms:
    \[\Vert \pi(f) \Vert \leq \sup_{\rho \in S} \Vert \rho (f)\Vert\]
\end{definition}
We will also need another definition of weak containment, which is equivalent to the one above. We recall that if $\pi$ is a representation of $G$, and $u,v \in H_\pi$, then the function $c_{u,v}:G \rightarrow \mathbb{C}$ defined by $g \rightarrow \langle \pi(g)u,v \rangle$ is called the \textit{matrix coefficient} of $\pi$ associated with $u$ and $v$.
\begin{definition}[Second Definition of Weak Containment]\label{SecDefWeakCont}
    Let $S$ be a set of representations of $G$, and $\pi$ another representation. Then $\pi$ is weakly contained in $S$ if every matrix coefficient of $\pi$ can be approximated, uniformly on compact subsets of $G$, by finite sums of matrix coefficients of representations from $S$.
\end{definition}
The notion of weak containment yields one way to define the Fell topology, which will be the most convenient one for us.
\begin{definition}[Fell Topology]
    A subset $S$ of $\hat{G}$ is closed in the Fell topology if it contains all irreducible representations that are weakly contained in it.
\end{definition}
A locally compact group is called \textit{Type I} if its unitary dual, endowed with the Fell topology, is $T_0$. If it is furthermore $T_1$, the group is called \textit{CCR}. Glimm \cite{Glimm} showed that the $T_1$ separation property of the unitary dual can be detected by the image of $C^*(G)$ under irreducible representations:
\begin{theorem}\label{ccrCompact}
    $G$ is CCR if and only if for every $f \in L^1(G)$ and every $\pi \in \hat{G}$, one has that $\pi(f)$ is a compact operator.
\end{theorem}
In a non-formal phrasing, Type I groups are those in which we can hope to understand their (unitary) representation theory via their unitary dual. More precisely, for a unitary representation $\pi$, one defines its \textit{support} (denoted $\text{supp}(\pi)$) to be the set of irreducible representations that are weakly contained in it. For Type I groups, there is a canonical way to define a measure on the unitary dual that will yield the representation as a direct integral of irreducible ones. The support of a representation is exactly the support of this measure, and the representation is weakly equivalent to its support.

We will now define two important families of representations.
\begin{definition}[Discrete Series, Tempered Representations]
    Let $\pi$ be an irreducible unitary representation of $G$.
    \begin{enumerate}
        \item $\pi$ is said to be in the \textbf{discrete series} if it is a sub-representation of the left regular representation.
        \item $\pi$ is said to be \textbf{tempered} if it is weakly contained in the left regular representation.
    \end{enumerate}
\end{definition}
Note that the definition above is not asymmetric: we could have used the right regular representation instead, for the left and right regular representations are equivalent. The set of irreducible tempered representations, $\text{supp}(\lambda)$, is by definition a closed subset of $\hat{G}$. It is called the \textit{tempered dual}, and denoted by $\hat{G}_\text{temp}$. The notion of temperedness extends to reducible representation as well: a representation $\pi$ is said to be tempered if it is weakly contained in the left regular representation, or equivalently if every representation in its support is tempered.

In many cases, it is useful to identify tempered representations as those representations whose matrix coefficients obey certain integrability conditions. For this, we need the following:
\begin{definition}
       A representation $\pi$ is said to be an $L^p$ representation if there exists a dense subspace $H_0\subset H_\pi$, such that for all $\xi, \eta \in H_0$ the matrix coefficient function $c_{\xi,\eta}$ is in $L^p(G)$.
\end{definition}
\begin{theorem}\cite[Theorems~14.1.1-14.1.2]{Dixmier}\label{DiscChar}
    Assume $G$ is unimodular. An irreducible representation $\pi$ is in the discrete series if and only if it is an $L^2$ representation.
\end{theorem}
Another standard name for discrete series representations is \textit{square-integrable representations}, which makes sense because of theorem \ref{DiscChar}. Having theorem \ref{DiscChar} in hand, one can hope to have a similar statement with respect to tempered representations, that is, one can expect tempered representations to have matrix coefficients that are "almost" square integrable. One direction of this statement is indeed true for general locally compact groups. It is a theorem due to Cowling, Haagerup and Howe \cite[Theorem~1]{CowHowHaa}.
\begin{theorem}[Cowling-Haagerup-Howe]
    Let $\pi$ be a representation such that there exists a cyclic vector $\xi \in H_\pi$ whose associated diagonal matrix coefficient $g \rightarrow \langle \pi(g)\xi , \xi \rangle$ is in $L^{2+\epsilon}$ for every $\epsilon > 0$. Then $\pi$ is weakly contained in the left regular representation.
\end{theorem}

$G$ is \textit{Amenable} if and only if all of its irreducible representations are tempered, which is equivalent to the trivial representation being tempered. $G$ has \textit{Kazhdan's Property (T)} if the trivial representation is an isolated point in $\hat{G}$. Property (T) can be phrased in terms of spectral gap: the gap in the spectrum refers here to the isolation of the trivial representation. For groups that don't admit property (T), that is, their full spectrum (unitary dual) don't have a gap, spectral gap is still an interesting property of representations.
\begin{definition}[Spectral Gap]
    A representation $\pi$ is said to have a spectral gap if it doesn't weakly contain the trivial representation.
\end{definition}
We note that many authors define spectral gap in a slightly different way: a representation is said to have a spectral gap whenever the trivial representation is isolated in its support. For us, it is more convenient to eliminate the invariant vectors if exist, and use the definition above.
\subsection{Smooth and Admissible Representations}
In this section, $G$ is assumed to be a locally compact, totally disconnected group (abbreviated tdlc). The basic structural property of tdlc groups is van Dantzig's theorem, which states that a tdlc group admits a basis of neighborhoods of the identity that consists of compact open subgroups.

For tdlc groups, it is convenient to consider a different class of representations, which we recall here. In this subsection, representations are no longer assumed to be unitary unless otherwise stated. For an introduction to smooth and admissible representations, the reader is referred to \cite{tadicClassical}, and also to Bernstein-Zelevinsky \cite{BZ}, where much of the theory was developed. These references are the main source from here until the end of the preliminaries. 
\begin{definition}[Smooth Representations, Admissible Representations]
     Let $\pi$ be a representation of $G$ on a complex vector space $V$.
    \begin{enumerate}
        \item A vector $v\in V$ is called \textbf{smooth} for the representation $\pi$ if its stabilizer in $G$ is open. Evidently, $v$ is smooth if and only if there exists a compact open subgroup $K\subset G$ such that $\pi(k)v=v$ for all $k\in K$.
        \item The set of smooth vectors in $V$ is denoted by $V^\infty$. $\pi$ is called \textbf{smooth} if $V=V^\infty$.
        \item $\pi$ is called \textbf{admissible} if for each compact open subgroup $K$, the set of $K$-fixed vectors, $V^k$, is finite dimensional.
    \end{enumerate}
\end{definition}
Smooth representations are easier to handle, and we do not lose much by considering them, as the following elementary lemma guarantees. 
\begin{lemma}
    Let $\pi$ be a representation on a Hilbert space $V$, such that for each $v\in V$ the map $g \rightarrow \pi(g)v$ is continuous. Then, $V^\infty$ is a dense invariant subspace of $V$.
\end{lemma}
\begin{proof}
    The fact that $V^\infty$ is an invariant subspace is immediate (and of course does not require any topology from $V$). Now let $v$ be in $V$, and $U$ be a closed convex neighborhood of $v$. Then, from continuity one has that there exists a compact open subgroup $K$ such that $\pi(K)v\subset U$. Then, $w=\int_K \pi(k)v \,d\mu k$ is a smooth vector invariant by $K$ in $U$.
\end{proof}
We recall that for general locally compact groups, the category of unitary representations of the group is equivalent to the category of non-degenerate *-representations of the Banach *-algebra $L^1(G)$. For tdlc groups, we can restrict ourselves to a very convenient class of test functions in $L^1(G)$. For a compact open subgroup $K$, let $H_K(G)$ be the algebra (with convolution, after fixing a Haar measure on the group) of compactly supported $K$-bi-invariant functions on G. It is an idempotented algebra, with idempotent $e_K$, the indicator of $K$.

We denote by $H(G)$ the algebra of locally constant, compactly supported functions on $G$, which is the union of $H_K(G)$ as $K$ runs over all compact open subgroups of $G$. $H(G)$ is called the \textit{Hecke algebra} of $G$ (note that some authors define the Hecke algebra to be an algebra of distributions on the group. As long as we fix a Haar measure on our group this is equivalent, and we will work with functions). The category of smooth $G$-representations is equivalent to the category of non-degenerate $H(G)$-modules.

A representation ($\pi, V$) is associated with an $H(G)$-module, in the standard way one associates an algebra representation of $L^1(G)$ to a representation of the group. If the representation is smooth, the associated module is non-degenerate. The representation is moreover associated with an $H_K(G)$-module for each $K$, and for $f\in H_K(G)$, $\pi(f)$ is an operator whose range is in the space of $K$-invariant vectors.
\begin{lemma}
    We have the following relations between representations of a tdlc group and of its Hecke algebra:
    \begin{enumerate}
        \item Let $(\pi, V)$ be a smooth representation of $G$. Then, $\pi$ is algebraically irreducible if and only if for every compact open $K$, the representation of $H_K(G)$ on the space $V^K$ of $K$-invariant vectors is algebraically irreducible.
        \item Let $(\pi, V)$ be a continuous unitary representation of $G$ on a Hilbert space. Then $\pi$ is topologically irreducible if and only if for every compact open $K$, the representation of $H_K(G)$ on the space $H^K$ of $K$-invariant vectors is topologically irreducible. 
    \end{enumerate}
\end{lemma}
\begin{proof}
    For (1), first assume that $V^K$ is an irreducible $H_K(G)$ module for all $K$. Then, as $V$ is the union of all of these spaces, it is irreducible as an $H(G)$-module (if $v,w \in V$, assume $v \in V^{K_1}$ and $w \in V^{K_2}$; so both are in $V^{K_1 \cap K_2}$, and hence an element of $H_{K_1\cap K_2}(G)$ takes $v$ to $w$). In the other direction, assume $V$ is an irreducible $H(G)$-module. Choose $K$ such that $V^K \ne 0$, and suppose that $W$ is a non-zero invariant $H_K(G)$ sub-module of $V^K$. Then, for $w \in W$ and $v \in V^K$, we have some $f \in H(G)$ such that $\pi(f)w=v$. Then, $v=\pi(e_K)\pi(f)\pi(e_k)w$, and $e_k \ast f \ast e_K \in H_K(G)$.

    Now, for (2), assume first that $\pi$ is topologically irreducible. Let $K$ be a compact open subgroup such that $V^K \ne 0$ and suppose that $W \subset V^K$ is a closed, $H_K(G)$-invariant subspace. Let $w \in W$ and $v \in V^K$. Topological irreducibility together with the fact that $H(G)$ is dense in $L^1(G)$ implies the existence of a sequence of elements $f_n \in H(G)$ such that $\pi(f_n)w \rightarrow v$; and by replacing $f_n$ with $e_K \ast f_n \ast e_K$ we may assume that $f_n \in H_K(G)$. But $\pi(f_n)w \in W$ for every $n$, and since $W$ is closed we get that $v \in W$. In the other direction, assume that for every $K$, the representation of $H_K(G)$ on $V^K$ is topologically irreducible. Suppose that $W$ is a closed invariant subspace of $V$. Let $w \in W^\infty$ (the smooth vectors are dense in the representation restricted to $W$) and $v \in V^\infty$. There exists some $K$ such that $w, v \in V^K$. Because the representation of $H_K(G)$ on $V^K$ is topologically irreducible, we can find a sequence $f_n \in H_K(G)$ such that $\pi(f_n)w \rightarrow v$. Now, the fact that $W$ is closed implies that $v \in W$. So, $V^\infty \subset W$, and since $V^\infty$ is dense in $V$ we get that $V = W$.
\end{proof}
\begin{corollary}
    Let $(\pi, H)$ be a unitary, admissible, topologically irreducible representation of $G$. Then, the underlying smooth representation $(\pi, H^\infty)$ is algebraically irreducible. 
\end{corollary}
\begin{proof}
    Let $K$ be a compact open subgroup such that $H^K \ne 0$. Then, $H^K$ is topologically irreducible by the lemma, but it is moreover algebraically irreducible because it is finite dimensional. As $H^\infty$ is the union of $H^K$ under all $K$, it follows again from the lemma that $H^\infty$ is algebraically irreducible.
\end{proof}
The next result shows that the category of smooth admissible representations can't be avoided if one wants to study the unitary dual of a tdlc group whose representation theory behaves nicely. It says that the unitary dual of CCR tdlc groups live in a bigger space, i.e., in the admissible dual.
\begin{proposition}\label{tdlc_ccr}
    Let $G$ be a tdlc group. Then, $G$ is CCR if and only if $G$ has the property that all of its irreducible unitary representations are admissible.
\end{proposition}
\begin{proof}
    First assume that $G$ is CCR. Let $\pi$ be an irreducible, unitary representation of $G$, and $K$ be a compact open subgroup. Then, for $e_K \in L^1(G)$, $\pi(e_K)$ is on the one hand compact (by theorem \ref{ccrCompact}) and on the other hand a projection. Hence, its range is finite dimensional, and it is exactly the set of $K$-invariant vectors. So $\pi$ is admissible. In the other direction, assume that every irreducible unitary representation of $G$ is admissible, and let $\pi$ be an irreducible unitary representation. Then, for $f$ in $L^1(G)$, $f$ is the limit of functions from $H(G)$. Hence, $\pi(f)$ is the norm limit of operators of finite rank, and is therefore compact.
\end{proof}
\subsection{Unitary dual of Reductive p-adic Groups}
From here until the end of the section, $k$ is a non-archimedean local field of characteristic zero (that is, a finite extension of the field of p-adic numbers), $\textbf{G}$ is a reductive algebraic group defined over $k$, and $G=\textbf{G}(k)$ is its group of $k$-rational points. We always denote by bold letters algebraic $k$-groups, and by regular letters the locally compact group of $k$-rational points of this algebraic group. We will quote two classical results that complete the previous subsections. The first is due to Harish-Chandra, and it tells us that for semisimple groups an inverse direction to Cowling-Haagerup-Howe's theorem also holds.
\begin{theorem}[Harish-Chandra]\label{tempMatCoef}
    Assume $\textbf{G}$ is not only reductive but semisimple. A representation $\pi$ is tempered if and only if it is an $L^{2+\epsilon}$ representation for every $\epsilon>0$. 
\end{theorem}
Theorem \ref{tempMatCoef} also holds in the archimedean case. The next crucial result is due to Bernstein.
\begin{theorem}[\cite{Bernstein}]
    Let $\pi$ be an irreducible, unitary representation of $G$. Then $\pi$ is admissible.
\end{theorem}
Bernstein's result first of all shows that reductive groups over non-archimedean local fields are CCR (see proposition \ref{tdlc_ccr}). But furthermore, it suggests a way to find the unitary dual of such groups. As we can identify $\hat{G}$ with a subset of the smooth, admissible dual (by passing to the smooth part of the representation), the standard strategy is as follows: classify all irreducible, admissible, smooth representations of $G$, and then detect which of them admits an invariant inner-product. The first part of this is done via the Langlands classification, which is our last goal in these preliminaries.

\subsection{Parabolic Induction and the Langlands Classification}
We fix a maximal split torus $\textbf{A}_0$ and a minimal parabolic $\textbf{P}_0=\textbf{M}_0\textbf{N}_0$, where $\textbf{M}_0=\textbf{Z}(\textbf{A}_0)$ is a Levi subgroup. A pair $(\textbf{P}, \textbf{A})$ of a split torus $\textbf{A}$ and a parabolic subgroup $\textbf{P}$ with Levi decomposition $\textbf{P}=\textbf{MN}$, $\textbf{M}=\textbf{Z}(\textbf{A})$ is called standard if $\textbf{P} \supset \textbf{P}_0$ and $\textbf{A} \subset \textbf{A}_0$. We denote by $X^\ast(\textbf{H})_k$ the group of $k$-rational characters of a $k$-algebraic group $\textbf{H}$. If $\textbf{H}$ is reductive, we define:
\[
^0H = \bigcap_{\alpha\in X^\ast(\textbf{H})_k}\text{ker}(\abs{\alpha}_k)
\]
Where $\abs{\cdot}_k$ is the absolute value on the local field $k$. The subgroup $^0H$ is open, normal, and the quotient $H / ^0H$ is free abelian and finitely generated; $^0H$ contains all compact subgroups of $H$ as well as the derived subgroup. A character (that is, a homomorphism of topological groups into $\mathbb{C}^*$) of $H$ is called \textit{unramified} if it factors through $H/^0H$. 

The general philosophy of Harish-Chandra is that admissible representations of $G$ may be understood in terms of representations of Levi factors of parabolic subgroups of $G$. Let $P=MN$, $M=Z(A)$ be a parabolic subgroup, and consider irreducible, admissible representations of $M$. Usually we will think of an irreducible admissible representation of $M$ in the form $\sigma \chi$, where $\sigma$ is (the smooth part of) some irreducible, unitary representation of $M$ and $\chi$ is an unramified character of the torus $A$. It is convenient to identify the unramified characters of a torus as elements in a vector space. To this end, we define the real and complex dual to the Lie algebra of a split torus $\textbf{A}$, whose centralizer is a Levi $\textbf{M}$, in the following way:
\[
\mathfrak{a}^\ast = X^\ast(\textbf{A})_k \otimes \mathbb{R} =X^\ast(\textbf{M})_k \otimes \mathbb{R}
\]
\[
\mathfrak{a}^\ast_\mathbb{C} = X^\ast(\textbf{A})_k \otimes \mathbb{C} =X^\ast(\textbf{M})_k \otimes \mathbb{C}
\]
Now, $\mathfrak{a}$ and $\mathfrak{a}_\mathbb{C}$ are defined as the dual spaces corresponding to these vector spaces. In order to identify elements of $\mathfrak{a}^\ast_\mathbb{C}$ as (unramified) characters of the torus $A$, we define Harish-Chandra's function $H:M\rightarrow \mathfrak{a}$ by:
\[
q^{<\chi,H(x)>}=\abs{\chi(x)}_k
\]
for all $\chi \in X^\ast(\textbf{A})_k$ and $x \in M$, where $q$ is the number of elements in the residue field of $k$. Now, with each $\nu \in \mathfrak{a}^\ast_\mathbb{C}$ we associate a character of $A$, $\chi_\nu$, which is defined as: $\chi_\nu(a)=q^{<H(a),\nu>}$. The association above identifies the abelian group of unramified characters of $A$ with a quotient of the vector space $\mathfrak{a}^\ast_\mathbb{C}$.
\begin{definition}[Parabolic Induction]
    Given a parabolic subgroup $P=MN$ and a smooth, admissible representation of $M$, $\sigma$, we define $i_{GM}(\sigma)$ to be $\text{Ind}_{P}^{G}(\sigma \delta_P^{\frac{1}{2}})$, where Ind is the (smooth) induction functor and $\delta_P$ is the modular function of $P$.  
\end{definition}
The parabolic induction functor $i_{GM}$ takes admissible representations of $M$ to admissible representations of $G$. The normalization factor $\delta_P^{1/2}$ ensures that this functor also takes unitary representations to unitary representations. It has a natural left adjoint.
\begin{definition}[Jacquet Functor]
    Given a parabolic subgroup $P=MN$ and an admissible representation $(\rho, V)$ of $G$, define the subspace $V_N$ = span of $\{u-\rho(n)u \mid n \in N\}$. As $M$ normalizes $N$, it perseveres the subspace $V_N$. We define $r_{MG}(\rho)$ to be the quotient of $\rho$ on the space of co-invariants $V/V_N$ restricted to $M$, tensored with $\delta_P^{-1/2}$.
\end{definition}
\begin{theorem}[Frobenius Reciprocity]\label{Frobenius}
    $r_{MG}$ is left adjoint to $i_{GM}$.
\end{theorem}
\begin{definition}
    An admissible representation $(\pi, V)$ is called \textit{supercuspidal} if for every proper parabolic subgroup $P=MN$, we have $r_{MG}(\pi)=0$.
\end{definition}
Supercuspidal representations are the ones that don't occur as sub-quotients of parabolically induced representations. An important result of Harish-Chandra detects supercuspidal representations via their matrix coefficients. This result shows that supercuspidal representations are always in the discrete series, and in particular are tempered.
\begin{theorem}[Harish-Chandra]
    A representation $\pi$ is supercuspidal if and only if its matrix coefficients are compactly supported when restricted to $^0G$.
\end{theorem}
The machinery of parabolic induction yields all representations that are not supercuspidal, and supercuspidal representations are always tempered. One celebrated result about parabolic induction is the Langlands classification theorem, which classifies all irreducible admissible representations of a reductive group in terms of  tempered representations of Levi subgroups of it. In the following definition, $\left( ,\right)$ is the pairing of $\mathfrak{a}^\ast$ with $\mathfrak{a}$ for some split torus $\textbf{A}$, and for a root $\alpha$ its co-root is denoted by $\alpha^\vee$.
\begin{definition}[Langlands Data]
    A \textbf{Langlands data set} is a triplet $(P, \sigma, \chi_\nu$) where: \begin{enumerate}
        \item ($P, A$) is a standard pair, with $P=MN$ and $M=Z(A)$.
        \item $\sigma$ is the smooth part of a (unitary) tempered representation of $M$.
        \item $\chi_\nu$ is an unramified character of $A$, associated to an element $\nu \in \mathfrak{a}^\ast$ satisfying $(\nu , \alpha^\vee) > 0$ for every root  $\alpha$ of $A$ that appears in $N$.
    \end{enumerate}
\end{definition}
\begin{theorem}[Langlands Classification]
    The set of irreducible, admissible representations of $G$ stands in one-to-one correspondence with Langlands data sets, in the following way:
    \begin{enumerate}
        \item Given a Langlands data set $(P, \sigma, \chi_\nu)$, the parabolically induced representation $i_{GM}(\sigma \chi_\nu)$ admits a unique irreducible quotient, denoted $J_{P, \sigma, \chi_\nu}$, that is called the \textit{Langlands quotient} of the data set.
        \item The assumption of Langlands quotient to a data set is one-to-one.
        \item Every irreducible admissible representation of $G$ is a Langlands quotient associated with some (unique, by (2)) data set.
    \end{enumerate}
\end{theorem}
The Langlands classification in the p-adic case is proved in \cite[Chapter~XI.2]{BorelWallach}. The Langlands classification tells us that (the smooth part of) every non-tempered unitary representation of $G$ is $J_{P,\sigma,\chi}$ for some Langlands data $(P,\sigma,\chi)$. These representations are known as the \textit{complementary series} of representations (sometimes the trivial representation is excluded). We can define more specifically:
\begin{definition}[Complementary Series Associated With $\sigma$]
    Given a standard pair $(P,A)$, with $P=MN$ and $M=Z(A)$, and a tempered representation $\sigma$ of $M$, the set of Langlands quotients $J_{P,\sigma,\chi}$ for $\chi$ such that $(P,\sigma,\chi)$ is a Langlands data set and the Langlands quotient is unitarizable is called \textit{the complementary series associated with $\sigma$}.
\end{definition}
\section{Spectral Gap Absorption}\label{sec3}
Here we formulate the two main theorems and explain how to deduce the first as a corollary from the second. We fix the following notations: $G = \prod \textbf{G}_i(k_i)$, where for each $i$, $k_i$ is a local field of characteristic zero, and $\textbf{G}_i$ is an almost simple, connected and simply connected algebraic $k_i$-group. We furthermore assume that $G_i=\textbf{G}(k_i)$ is not compact, for every $i$.
\begin{theorem}\label{main theorem}
    Let $\pi$ and $\rho$ be two unitary representations of $G$, such that $1 \nprec \pi$. Then $1 \nprec \pi \otimes \rho$. That is, having a spectral gap is preserved under tensor products. 
\end{theorem}
Before continuing, it is in place to explain what happens for non-simply connected groups. Theorem \ref{main theorem} is false in the non-simply connected case, but its veracity in the simply connected case enables us to fully understand what happens in the other cases. The source of the problem in the non-simply connected case is that $G=\textbf{G}(k)$, for $k$ that is not algebraically closed, may have characters. If $\rho \in \hat{G}$ is a non-trivial character, then although it has a spectral gap, tensoring it with $\rho^{-1}$ yields the trivial representation. We can also get representations with almost invariant vectors if our group doesn't have property (T) - if it is rank one p-adic, or rank one Lie group locally isomorphic to $SO(n,1)$ or $SU(n,1)$ - we simply take a representation $\pi$ with almost invariant vectors, and tensor the non-trivial character $\rho$ with $\pi \otimes \rho^{-1}$. The next corollary allows us to understand what sort of spectral gap obeys an absorption principle for non-simply connected groups, and what sort of spectral gap doesn't obey this principle. For simplicity reasons, we explain only the case of simple groups. Let $\textbf{G}$ be a $k$-isotropic, almost simple algebraic $k$-group that is not necessarily simply connected, $\Tilde{\textbf{G}}$ its simply connected cover, and $f:\Tilde{\textbf{G}} \rightarrow \textbf{G}$ the central $k$-isogeny. For a $k$-algebraic group $\textbf{H}$, we denote by $\textbf{H}(k)^+=H^+$ the group generated by unipotent radicals of parabolic subgroups, which is the same (over a perfect field) as the subgroup generated by unipotent elements \cite[I.1.5.2]{Margulis}. $H^+$ is a normal, closed subgroup. Now, since $\Tilde{\textbf{G}}$ is simply connected, $\Tilde{G}^+=G$ \cite[Theorem~I.2.3.1]{Margulis}. Moreover, the isogeny $f$ takes $\Tilde{G}^+=\Tilde{G}$ onto $G^+$ \cite[Proposition~I.1.5.5]{Margulis}, and $G/G^+$ is a finite abelian group \cite[Theorem~I.2.3.1]{Margulis}. It turns out that a spectral gap obeys the absorption principle if and only if it comes from $G^+$.
\begin{corollary}
    Let $G$ be the group of $k$-points of an almost simple, connected algebraic $k$-group that is not necessarily simply connected. Let $\pi$ be a representation such that $1 \nprec \pi$. Then, $1 \nprec \pi \otimes \rho$ for every representation $\rho$ if and only if $\pi|_{G^+}$ has a spectral gap as well.
\end{corollary}
\begin{proof}
    If $1 \nprec \pi|_{G^+}$, then by pulling back the representation to the simply connected cover, we get the absorption from theorem \ref{main theorem}. Assume now that $1 \prec \pi|_{G^+}$. Then, induction on both sides of the equation yields:
    \[
    \text{Ind}{G^+}^G(1) \prec \text{Ind}_{G^+}^G(\pi)
    \]
    \[
    L^2(G/G^+) \prec \pi \otimes L^2(G/G^+)
    \]
    But, since $G/G^+$ is finite, $L^2(G/G^+)$ contains the trivial representation. So we get:
    \[
    1 \prec \pi \otimes L^2(G/G^+)
    \]
    and hence the spectral gap of $\pi$ is not preserved when tensoring with $L^2(G/G^+)$.
\end{proof}

We will deduce theorem \ref{main theorem} from a certain filtration of the unitary dual of $G$, given by integrability parameters of matrix coefficients. For this we need some definitions. 
\begin{definition}[Closed Ideal]
    A set $S \subset \hat{G}$ is called a \textbf{closed ideal} if for all $\pi \in S$ and $\rho \in \hat{G}$, $\text{supp}(\pi \otimes \rho) \subset S$.
\end{definition}
Our strategy is to find sufficiently many closed ideals in the unitary dual of $G$, such that every closed subset of the dual that doesn't contain the trivial representation will lie in such an ideal. Then we will deduce our theorem, using the following easy observation:
\begin{lemma}\label{supp of tensor}
    Let $H$ be a locally compact group, and $S$ be a closed ideal in the unitary dual of $H$. Then, for every representation $\pi$ such that $\text{supp}(\pi) \subset S$ and every representation $\rho$, $\text{supp}(\pi \otimes \rho) \subset S$.
\end{lemma}
\begin{proof}
    To deduce this from the definition of a closed ideal, we note that we can identify the support of $\pi \otimes \rho$ using the supports of $\pi$ and of $\rho$. That is, one has:
    \[\text{supp}(\pi \otimes \rho) = \overline{\bigcup_{\substack{\pi' \in \text{supp}(\pi) \\ \rho' \in \text{supp}(\rho)}} \text{supp}(\pi' \otimes \rho')}\]
    The right-hand side is clearly contained in the left-hand side. To see the other inclusion, we need to use the definition of weak containment via approximations of matrix coefficients (definition \ref{SecDefWeakCont}). We need to show that $\pi \otimes \rho$ is weakly contained in the union that appears on the right. Since each representation is weakly equivalent to its support, it is enough to show weak containment in the set $\{\pi' \otimes \rho' \mid \pi' \prec \pi, \rho' \prec \rho\}$ (instead of passing to the union of the supports). Let $f$ be some matrix coefficient of $\pi \otimes \rho$. Then, it is a sum of the form $\sum h_ig_i$, where $h_i$ are matrix coefficients of $\pi$ and $g_i$ are matrix coefficients of $\rho$. $h_i$ can be approximated, uniformly on compacta, by sums of matrix coefficients of representations from $\text{supp}(\pi).$ Similarly, $g_i$ can be approximated in that way by matrix coefficients of representations from $\text{supp}(\rho)$. Then the products will converge to $h_ig_i$, and hence $h_ig_i$ can be approximated by matrix coefficients of representations of the form $\pi' \otimes \rho'$, where $\pi' \in \text{supp}(\pi)$ and $\rho' \in \text{supp}(\rho)$. This shows that the left-hand side is weakly contained in the right-hand side, and hence contained in its closure.
\end{proof}
\begin{definition}\label{p_ideals}
    $\hat{G}_p$ is the set of all irreducible representations that are $L^{p+\epsilon}$ representations for all $\epsilon>0$.
\end{definition}
As matrix coefficient functions are bounded on $G$, we have that $\hat{G}_p \subseteq \hat{G}_q$ for $p \le q$. We now formulate our second main theorem. Whereas the spectral gap absorption theorem \ref{main theorem} holds for semisimple groups, the next theorem only applies to almost simple groups. Namely, for the following theorem $G=\textbf{G}(k)$ is the group of $k$-points of a an almost simple, simply connected algebraic $k$-group, for some field  $k$ of characteristic zero, and that $G$ is not compact.  
\begin{theorem}\label{filtration}
   Let $S \subset \hat{G}$ be a closed subset such that $1 \nprec S$. Then, $\exists p < \infty$ such that $S \subset \hat{G}_p$. 
\end{theorem}
The next section of the paper is devoted to the proof of theorem \ref{filtration}. To deduce theorem \ref{main theorem}, we need a result from \cite{deLaat-Siebenand}, following an earlier work of Samei and Wiersma \cite{Samei-Wiersma}. A locally compact group $H$ is called \textit{Kunze-Stein} if it satisfies the equation $L^p(H) \ast L^2(H) \subset L^2(H)$, for $1 \leq p < 2$.
\begin{theorem}\cite[Theorem~1.1]{deLaat-Siebenand}\label{deLaat-Siebenand}
    Let $H$ be a Kunze-Stein group. Then, $\hat{H}_p$ is a closed ideal in $\hat{H}$.
\end{theorem}
The fact that if $\pi$ is an $L^{q}$-representation for some $q$, then $\pi \otimes \rho$ is also an $L^q$-representation is immediate. Hence, the novelty in theorem \ref{deLaat-Siebenand} is the fact that the sets $\hat{H}_p$ are closed. We are now ready to deduce theorem \ref{main theorem} from theorem \ref{filtration}.
\begin{proof}[proof of theorem \ref{main theorem}]
    First, we assume that $G$ is simple. Let $\pi$ and $\rho$ be unitary representations of $G$ such that $1 \nprec \pi$. By theorem \ref{filtration}, there exists a $p$ such that $\text{supp}(\pi) \subset \hat{G}_p$. 
    Simple algebraic groups over local fields are Kunze-Stein: this is a result of Cowling in the archimedean case \cite{Cowling} and of Vecca in the non-arhcimedean case \cite{Veca}.
    So, by theorem \ref{deLaat-Siebenand}, $\hat{G}_p$ is a closed ideal. Using lemma \ref{supp of tensor}, we get that $\text{supp}(\pi \otimes \rho) \subset \hat{G}_p$. But $G$ is not compact, so $1 \notin \hat{G}_p$, and we get that $1 \nprec \pi \otimes \rho$.

    We now prove the theorem in the general case, i.e. when $G=\prod_{i=1}^n G_i$ is a semisimple group, and each $G_i$ is some non-compact simple algebraic group over a local field of characteristic zero, for which we have already proved the theorem. The arguments are the same as the ones that Bekka gives in \cite[Lemma~4]{BekkaInvMean}. Let $\pi$ be some unitary representation of $G$ that has a spectral gap, and $\sigma$ any other unitary representation. Assume towards contradiction that $1 \prec \pi \otimes \sigma$. This means that there exists some sequence of irreducible representations $\rho_j$ in the support of $\pi \otimes \sigma$ that converges to the trivial representation in Fell topology. Applying lemma \ref{supp of tensor}, we may replace the irreducible representations $\rho_i$ with the sequence of (not irreducible) representations $\pi_j \otimes \sigma_j$, where $\pi_j$ is some irreducible representation in $\text{supp}(\pi)$ and $\sigma_j\in \text{supp}(\sigma)$. We recall that every irreducible representation of $G$ is of the form $\tau_1 \times ... \times \tau_n$, where $\times$ denotes the outer tensor product and $\tau_i$ is an irreducible representation of $G_i$ (this is the case since the $G_i$'s are Type I, see \cite[Theorem~7.17]{Folland}). Hence, $\pi_j=\pi_j^1 \times ... \times \pi_j^n$ and $\sigma_j = \sigma_j^1 \times ... \times \sigma_j^n$ for each $j$. We also have that $\pi_j \otimes \sigma_j = (\pi_j^1 \otimes \sigma_j^1) \times ... \times (\pi_j^n \otimes \sigma_j^n)$ for each $j$. Since the sequence $\pi_j^i \otimes \sigma_j^i$ converges to the trivial representation of $G_i$, for each $i$, we get that the sequence $\pi_j^i$ already converges to the trivial representation. Indeed, assume that $1_{G_i} \nprec \bigoplus \pi_j^i$. Then, by the spectral gap absorption principle for the simple group $G_i$, we also have $1_{G_i} \nprec \left( \bigoplus \pi_j^i \right) \otimes \left( \bigoplus \sigma_j^i\right)$. But $\bigoplus\pi_j^i \otimes \sigma_j^i \prec \left( \bigoplus \pi_j^i \right) \otimes \left( \bigoplus \sigma_j^i\right)$, and we saw that $1 \prec \bigoplus\pi_j^i \otimes \sigma_j^i$. Then, for each $i$, $\pi_j^i \rightarrow 1_{G_i}$, and hence $\pi_j^1 \times ... \times \pi_j^n \rightarrow 1_G$, and $1 \prec \pi$. This yields a contradiction. 
\end{proof}
\begin{remark}
    Theorem \ref{filtration} is already known in many cases. If $G$ is of higher rank, the theorem is proved in \cite{Hee}, after it was already known in the archimedean case by \cite{Cowling2}. For the rank one case, if $k$ is archimedean, this is \cite[Theorem~2.5.2]{Cowling2}. So, henceforth we focus on the non-archimedean, rank one case, and prove the the result there. For the rank one p-adic case, the result was proved by Bader and Sauer for quasi-split groups \cite[Theorem~4.5]{Bader-Sauer}. Our method of proof is similar to that of Bader and Sauer, and treats the non-quasi-split case as well. As the proof is uniform, we will not distinguish the (already known) quasi-split case.
\end{remark}
\begin{remark}
     The rank one p-adic groups were classified by Tits, and the classification is given in Appendix \ref{appen}. Though it may seem that we don't use the classification in the proof, we actually rely heavily on a result of Abe and Herzig \cite{Abe-Herzig} which uses the classification. Our proof is therefore too not independent of the classification.
\end{remark}

\section{Proof of theorem \ref{filtration}}\label{sec4}

In this section we prove theorem \ref{filtration}. First, we gather some more or less classical facts, mostly due to Harish-Chandra \cite[Chapter~5]{Silberger}, which help us control the complementary series associated with a representation of $M$, where $M$ is a Levi factor of a corank-1 parabolic. After that, we restrict our attention to the rank one case and prove the theorem.
\subsection{Complementary Series Coming From Corank-1 Parabolics}
In this subsection, $k$ is a non-archimedean local field of characteristic zero and $G=\textbf{G}(k)$ is the group of $k$-points of a semisimple algebraic $k$-group. To find the unitary dual of $G$, we need to answer the following question: given a parabolic $P=MN$ and an irreducible, admissible, unitary representation $\sigma$ of $M$, when is the Langlands quotient $J_{P,\sigma,\chi}$ unitary? The theory of intertwining operators suggests that a question that should be asked prior to that is the question of reducibility of the induced representation $i_{GM}(\sigma \chi)$. It is a fundamental fact that $i_{GM}(\sigma\chi)$ is irreducible for almost every unramified character (in the algebraic variety of unramified characters), and the points of reducibility are of great importance. To this aim, as the results in the corank-1 case are clearer and sufficient for us, we stick to the corank-1 case. That is: $A \subset A_0$ is a rank one split torus, $M=Z(A)$ its centralizer and $P=MN \supset P_0$ is the associated standard parabolic subgroup (we fix in advance some set of simple roots). Then, $P$ is called a \textit{corank-1 parabolic}. We take a supercuspidal, unitary representation $\sigma$ of $M$ and for each $\nu \in \mathfrak{a}^*_c$ we look at the parabolically induced representation $i_{GM}(\sigma \chi_\nu)$. The theory here works better with supercuspidal representations, and for us (in the rank one case) it is not an obstruction since $M$ is compact and all of its irreducible representations are supercuspidal. As $A$ is rank one, $\nu$ can be identified with a complex number: we fix a single non-zero element in $\mathfrak{a}^*_\mathbb{C}$ and consider it as a basis for the vector space. We will choose the element $\rho \in \mathfrak{a}^\ast$, which is half the sum of the roots that appear in $N$. The character $\chi_{\rho}$ is the character we use for the normalization of the induction, $\delta_P^{1/2}$, and every unramified character is of the form $\chi_{s\rho}=\delta_P^{(1/2)s}$ for some $s \in \mathbb{C}$. The complementary series associated with $\sigma$ is given by the unitarizable Langlands quotients $J_{P,\sigma,\chi_{s\rho}}$, for $s$ positive real. The Weyl group of the torus $A$ acts on representations of $M$, and we set $W_G(\sigma)=\{w \in W(G,A) \mid w(\sigma) = \sigma\}$. $W_G(\sigma)$ is either trivial or of the form $\{1,w_0\}$. The representation $i_{GM}(\sigma\chi_{s\rho})$ is hermitian (and hence may be unitarizable), for $s$ real, only if $W_G(\sigma)$ is non-trivial. $\sigma$ with a trivial stabilizer in the Weyl group contributes no complementary series, and hence we focus, from now on, on $\sigma$ such that $w_0(\sigma)=\sigma$. We define now the Plancherel measure $\mu_{G,\sigma}$, or Harish-Chandra's $\mu$-function, as in \cite[Section~3.1]{Abe-Herzig} (compare to the discussion in  \cite[Section~2]{Shahidi}). It follows from Harish-Chandra's work \cite[Chapter~5]{Silberger} that this function, as a function on $\mathbb{C}$ (in the notations of \cite{Abe-Herzig}, it is the function $s \rightarrow \mu^G(\sigma\chi_{s\rho})$), is meromorphic and its zeros and poles give precise information about the reducibility points of $i_{GM}(\sigma\chi_{s\rho})$. It is phrased concisely in \cite{Abe-Herzig}, so we use them as a reference. The following theorem is precisely \cite[Proposition~3.5]{Abe-Herzig} adjusted to our notations.
\begin{theorem}[Plancherel Measure and Reducibility of Parabolic Induction]\label{PlanMeas}
    There is a unique $s_0 \in [0,1]$ such that $i_{GM}(\sigma\chi_{s\rho})$ is reducible if and only if $s \in \{s_0,-s_0\}$. If $s_0=0$, $\mu_{G,\sigma}$ is holomorphic and non-vanishing on $\mathbb{R}$. If $s_0 > 0$, the function $\mu_{G,\sigma}$ on $\mathbb{R}$ has a double zero at $s=0$, simple poles at $s=\pm s_0$ and is holomorphic and non-vanishing otherwise. 
\end{theorem}
The reason we are interested in the reducibility points of the parabolic induction is that they determine the complementary series associated with $\sigma$. That is, the reducibility point is precisely the end of the complementary series. This is a standard fact. The proof of the next theorem is the same as the arguments in \cite[Theorem~8.1]{Shahidi} - the only difference is the location of the poles which is specific to the extra assumptions made there. 
\begin{proposition}\label{redPointUnit}
The reducibility point $s_0$ in theorem \ref{PlanMeas} is the end of the complementary series associated with $\sigma$. That is, if $s_0=0$, $\sigma$ contributes no complementary series; if $s_0 > 0$, then the Langlands quotients $J_{P,\sigma,\chi_{s\rho}}$ for $s > 0$ are unitarizable if and only if $s \leq s_0$
\end{proposition}
\begin{remark}\label{endIsUnitary}
    We note that the Langlands quotient at the end of the complementary series, i.e. at the point of reducibility, is always unitarizable. This fact will be crucial to us later. This follows directly from the theory of intertwining operators, which yields a bilinear form that is positive until the point of reducibility, where it becomes non-negative and positive on the Langlands quotient. In fact, a more general result is true, though much less trivial: All irreducible subquotients at the end of a complementary series are unitarizable. This is originally due to Mili\v ci\'c \cite{Milicic}, as explained in \cite[Section~3.c]{Tadic} 
\end{remark}
\subsection{Back to the Rank One Case}
From now until the end of the section, $k$ is a non-archimedean local field of characteristic zero and $\textbf{G}$ is a $k$-algebraic, $k$-rank one, almost simple, connected and simply connected group. $G=\textbf{G}(k)$. We fix a rank one split torus $A$ and a representative for the unique conjugacy class of proper parabolic subgroups, $P=MN$ where $M=Z(A)$. We proceed to prove theorem \ref{filtration}. The first piece of information that we need is that the integrability parameter of the matrix coefficients of a representation can be read from its location on the complementary series it belongs to. This is part of the theory of asymptotic behavior of matrix coefficients \cite[4.4]{Casselman}. This asymptotic behavior for a parabolically induced representation is determined by the central characters of the representation it is induced from. Specifically in our case, we have:
\begin{lemma}\label{IntCond}
    Let $\sigma \in \hat{M}$, and consider the Langlands quotient $\pi_{\sigma,s}=J_{P,\sigma,\chi_{s\rho}}$, for some $s \in (0,1)$. Assume that $\pi_{\sigma,s}$ is in the complementary series. Then, $\pi_{\sigma,s} \in \hat{G}_p$ if and only if $p \geq \frac{2}{1-s}$
\end{lemma}
\begin{proof}
    Let $f=c_{u,v}$ be some matrix coefficient of the irreducible representation $\pi_{\sigma,s}$. Then, when $G$ acts on the spaces of matrix coefficients of admissible representations of $G$ by translations, $f$ spans a sub-representation equivalent to $\pi_{\sigma,s}$. We denote by $\mathscr{X}_f(P,A)$ the set of exponents of this representation with respect to $(P,A)$, as in \cite{SilIntOfMatCoef}. A character $\chi \in X(A)$ is in $\mathscr{X}_f(P,A)$ if and only if it is the central character of a component of $\overline{r}_{MG}(\pi_{\sigma,s})$, where $\overline{r}_{MG}$ is the normalized Jacquet functor with respect to the opposite parabolic subgroup. For a smooth representation $\rho$, we denote by $\tilde{\rho}$ the smooth contragredient representation of $\rho$. We have Casselman's pairing: $\widetilde{\bar{r}_{MG}(\rho)} = r_{MG}(\widetilde{\rho})$ \cite[4.2.5]{Casselman}. Furthermore, we have $i_{GM}(\widetilde{\rho})=\widetilde{i_{GM}(\rho)}$ \cite[2.3]{BZ}. Hence, using Frobenius reciprocity we get:
    $
    \text{Hom}_G(i_{GM}(\sigma\delta_P^{(1/2)s}),\pi_{\sigma,s})=\text{Hom}_G(\widetilde{\pi_{\sigma,s}},\widetilde{i_{GM}(\sigma\delta_P^{(1/2)s}}))=\text{Hom}_G(\widetilde{\pi_{\sigma,s}},{i_{GM}(\widetilde{\sigma\delta_P^{(1/2)s}}}))=
    \text{Hom}_M(r_{MG}(\widetilde{\pi_{\sigma,s}}),\widetilde{\sigma \delta_P^{(1/2)s})}=\text{Hom}_M(\widetilde{\bar{r}_{MG}(\pi_{\sigma,s})},\widetilde{\sigma \delta_P^{(1/2)s}})=
    \text{Hom}_M(\sigma\delta_P^{(1/2)s},\bar{r}_{MG}(\pi_{\sigma,s}))
    $.
    So, $\sigma \delta_P^{(1/2)s}$ is always a component of $\bar{r}_{MG}(\pi_{\sigma,s})$. The geometric lemma \cite[Theorem~5.2]{BZ} tells us that $\overline{r}_{MG}(i_{GM}(\sigma \delta_P^{(1/2)s}))=\sigma \delta_P^{(1/2)s} \oplus \sigma \delta_P^{-(1/2)s}$. Hence, if $i_{GM}(\sigma \delta_P^{(1/2)s})$ is irreducible, $\bar{r}_{MG}(\pi_{\sigma,s})$ is of length two, with the other component being $\sigma \delta_P^{-(1/2)s}$. If it is reducible, then each summand has a non-zero Jacquet module, the induction is of length 2 as well, and we get that $\bar{r}_{MG}(\pi_{\sigma,s})=\sigma \delta_P^{(1/2)s}$. In both cases, the relevant central exponent exponent is $\delta_P^{(1/2)s}$. Now, the explicit computation is precisely \cite[Corollary~2.6]{SilIntOfMatCoef} and we get that $p \geq \frac{2}{1-s}$
\end{proof} 
Lemma \ref{IntCond} tells us that for a fixed representation $\sigma \in \hat{M}$, as we go along its complementary series, the Langlands quotients leave all ideals of the form $\hat{G}_p$ when the parameter $s$ approaches one. So, we need to control the length of different complementary series. The next observation is that in the simply connected case, the only complementary series that goes all the way to $s=1$ is the one associated with $\sigma=1$, i.e. the complementary series that converges to the trivial representation.
\begin{lemma}\label{UnifBound}
    Let $\sigma$ be a unitary representation of $M$. If the end of the complementary series associated with it is $s=1$ (that is, the reducibility point of theorem \ref{PlanMeas} is $s=1$), then $\sigma=1$.
\end{lemma}
\begin{proof}
    Let $\sigma \in \hat{M}$ be such that the reducibility point at the end of the complementary series is in $s=1$. Then, $J_{P,\sigma,\chi_\rho}$ is unitarizable, as ends of complementary series are always unitarizable (see remark \ref{endIsUnitary}). Now, by \cite[Theorem~XI.3.3]{BorelWallach} we have that uniformly bounded representations (and in particular unitarizable representations) are either tempered, or are of the form $J_{P,\sigma,\chi_{s\rho}}$ with $s$ strictly between $0$ and $1$, or have a non-compact kernel. So, for $J_{P,\sigma,\chi_{\rho}}$ to be uniformly bounded, it must have a non-compact kernel. The kernel, which is a normal subgroup, is either central (and hence compact) or contains the group $G^+$ generated by unipotent elements \cite[Theorem~I.1.5.6]{Margulis}. But, $\textbf{G}$ is a simply connected, $k$-isotropic and $k$-almost simple group, so $G$ coincides with $G^+$ \cite[Theorem~I.2.3.1]{Margulis}, hence there are no non-compact non-trivial normal subgroups, and $\sigma=1$.
\end{proof}
\begin{remark}
    Lemma \ref{UnifBound} is in fact the only place in the proof where we use the assumption that $G$ is simply connected. A non-simply connected group may exhibit a complementary series that reaches all the way to $s=1$, and at the end instead of the trivial representation, it will converge to a character of $G/G^+$.
\end{remark}
The last result that we need is the fact that for a fixed group, there are only finitely many possible ends of complementary series. We know from proposition \ref{redPointUnit} that in order to find the possible complementary series coming from supercuspidal representations of Levi factors of corank-1 parabolics, we should detect the poles of the Plancherel measure. This was done by Shahidi, for $G$ quasi-split and $\sigma$ generic \cite[Theorem~8.1]{Shahidi}. In the rank one case, the quasi-split groups are just $SL_2(k)$ and $SU_3(h)$ associated with a quadratic extension. This allowed Bader and Sauer to complete the proof of theorem \ref{filtration} in these cases \cite{Bader-Sauer}. If $G$ is not quasi-split, then it has a unique quasi-split inner form $G'$. The parabolic $P=MN$ corresponds in $G'$ to a corank-1 parabolic $P'=M'N'$. Now, if $n=rd$, and $D$ is a central division algebra of degree $d$ over $k$, we have the Jacquet-Langlands correspondence between discrete series representations of $GL_n(k)$ and its inner form $GL_r(D)$ \cite{DKV}. In case $G$ is rank one, the theory applies to $M$ and $M'$, since $^0M$ is a reductive p-adic group that is anisotropic, and such groups are always inner forms of $A_n$ \cite{Tits2} (after maybe passing to a restriction of scalars). The cuspidal representation $\sigma$ of $M$ whose end of complementary series we are trying to find corresponds to a discrete series representation $\sigma'$ of $M'$. Now, one can try to apply Shahidi's methods in $G'$ for $\sigma'$, or more accurately to the supercuspidal representation it is induced from. This allowed some authors to deal with specific cases \cite{MuicSavin, TadicDivAlg}. Abe and Herzig \cite{Abe-Herzig} were able to finish the cases of rank one groups using this machinery. They didn't determine the Plancherel measure precisely, but used a global argument to show that there are only finitely many possible poles in this case for a fixed group.
\begin{theorem}\cite[Remark~3.61]{Abe-Herzig}\label{Abe-Herzig}
    if $G$ is rank one, then there is a finite set of numbers in $[0,1]$ that may be ends of complementary series of representations. 
\end{theorem}
Using lemmas \ref{IntCond} and \ref{UnifBound} together with theorem \ref{Abe-Herzig}, we are now in a position to complete the proof of theorem \ref{filtration}.
\begin{proof}[proof of theorem \ref{filtration}]
    Let $A$ be a closed set in $\hat{G}$. In the finite set of possible ends of complementary series, promised by theorem \ref{Abe-Herzig}, let $s_0$ be the maximal element that is not one. Then, lemma \ref{IntCond} together with lemma \ref{UnifBound} imply that every representation that is not on the complementary series associated with $1 \in \hat{M}$ belongs to $\hat{G}_\frac{2}{1-s_0}$. Hence, if the set $A$ is not contained in any ideal of the form $\hat{G}_p$, it contains representations converging to the trivial representation, and therefore $1 \in A$.
\end{proof}

\section{The Restriction Map From a Semisimple Group to a Lattice}\label{sec5}

In this section, we demonstrate an application of the spectral gap absorption principle, by giving an affirmative solution to a conjecture of Bekka and Valette, given in \cite{BeV}.

\begin{conjecture}[Bekka-Valette]
    Let $G$ be a non-compact semisimple Lie group, and $\Gamma$ a lattice. Then, there exist unitary representations of $\Gamma$ that are not weakly contained in any representation restricted from $G$.
\end{conjecture}

This conjecture can be phrased in the following way: Does the restriction map from $\tilde{G}$ to $\tilde{\Gamma}$ have a dense image? (Here, $\tilde{H}$ is the set of all unitary representations of a locally compact, second countable group $H$ on separable Hilbert spaces, with the Fell topology). Bekka and Valette related this to a map from $C^*(\Gamma)$ to the multiplier algebra $M(C^*(G))$, which is injective if and only if the conjecture fails. In their paper, they showed the conjecture to hold for many groups: Groups with property (T), any lattice in $SL_2(\mathbb{R})$, any non-uniform lattice in $SL_2(\mathbb{C})$ and any arithmetic lattice in $SO(n,1)$ for $n \ne 3, 7$. We prove that the conjecture holds for a wider family of groups - namely, semisimple groups of the form $G=\prod \textbf{G}_i(k_i)$, where each $k_i$ is a local field of characteristic zero and each $\textbf{G}_i$ is an almost simple, connected and simply connected algebraic $k_i$-group that is $k_i$-isotropic. We note that for semisimple real algebraic groups, the connected component identifies with $G^+$, and hence the spectral gap absorption applies to any connected algebraic Lie group. For semisimple Lie groups with finite center, the spectral gap absorption still holds (see, for example, \cite[Lemma~4]{BekkaInvMean}).

\begin{theorem}\label{BekVal}
    Let $G$ a group that belongs to one of the following families:
    \begin{enumerate}
        \item Connected semisimple Lie groups with finite center.
        \item $G=\prod_{i=1}^n \textbf{G}_i(k_i)$, where each $k_i$ is a local field of characteristic zero and $\textbf{G}_i$ is an almost simple, connected and simply connected algebraic $k_i$-group, which is $k_i$-isotropic.
    \end{enumerate}
    Let $\Gamma$ be a lattice in $G$. Then, the image of the restriction map from $\tilde{G}$ to $\tilde{\Gamma}$ has a non-dense image in Fell topology. Moreover, every non-trivial finite dimensional representation of $\Gamma$ is not weakly contained in any representation restricted from $G$. 
\end{theorem}

To prove theorem \ref{BekVal}, we use two facts. The first is the spectral gap absorption principle, theorem \ref{main theorem}. The second thing that we need is the fact that $L^2_0(G/\Gamma)$ has a spectral gap.
\begin{definition}
    A lattice $\Lambda$ in a locally compact group $H$ is said to be \textit{weakly cocompact} if the $H$-representation $L^2_0(H/\Lambda)$ has a spectral gap.
\end{definition}
Any lattice $\Gamma$ in a group $G$ that satisfies the assumptions of theorem \ref{BekVal} is weakly cocompact. For Lie groups, it is proved in \cite[Lemma~3]{BekkaInvMean}. For arithmetic lattices in the algebraic case, it follows (after maybe decomposing $\Gamma$ to a product of irreducible lattices) from the work of Clozel \cite{Clozel}. In the rank-1 p-adic case, for non-arithmetic lattices, it follows from the fact that lattices in p-adic groups are uniform, and every cocompact lattice is weakly cocompact \cite[Corollary~III.1.10]{Margulis}. We are now ready to prove theorem \ref{BekVal}
\begin{proof}[proof of theorem \ref{BekVal}]
    Let $\pi$ be some non-trivial, irreducible, finite dimensional representation. Assume towards contradiction that $\pi$ is weakly contained in some representation $\sigma$ restricted from $G$ to $\Gamma$. We note that $1 \nprec \bar{\pi}$ for the contragredient representation $\bar{\pi}$ of $\pi$. But, being finite dimensional, $\pi$ and $\bar{\pi}$ are not weakly mixing: that is, $1 \le \pi \otimes \bar{\pi}$. Now, since $\pi \prec \sigma|_\Gamma$, we get that $1 \prec \bar{\pi} \otimes \sigma|_\Gamma$. Taking induction to $G$ on both sides of the equation, we get that $1 \le L^2(G/\Gamma)=\text{Ind}(1)\prec \text{Ind}(\bar{\pi} \otimes \sigma|_\Gamma)$. Now, $\text{Ind}(\bar{\pi} \otimes \sigma|_\Gamma)=\text{Ind}(\bar{\pi}) \otimes \sigma$ by \cite[Theorem~E.2.5]{BekHarVal}. So we have that $1 \prec \text{Ind}(\bar{\pi}) \otimes \sigma$. We now note that the spectral gap of $\bar{\pi}$ is inherited by the induced representation $\text{Ind}(\bar{\pi})$. This follows from the fact that the lattice $\Gamma$ is weakly cocompact in $G$ (that is, $1 \nprec L^2_0(G/\Gamma)$, by an argument of Margulis \cite[Proposition~III.1.11]{Margulis}. So, we have shown that although $\text{Ind}(\bar{\pi})$ has a spectral gap, $\text{Ind}(\bar{\pi}) \otimes \sigma$ doesn't have a spectral gap. This contradicts the spectral gap absorption principle for $G$, and hence such a $\sigma$ doesn't exist.
\end{proof}

\appendix
\section{the rank one p-adic groups}\label{appen}
We recall the classification of almost simple groups over local fields. A semisimple algebraic group over an algebraically closed field is determined up to central isogeny by its Dynkin diagram. To give the form of over the local field, we need to know two more things \cite{Satake, Tits2}:
\begin{enumerate}
    \item The semisimple anisotropic kernel, which is the derived subgroup of the Levi of a minimal parabolic.
    \item The Tits index: this is the Dynkin diagram, with certain vertices distinguished - the ones that vanish on the maximal split torus, and an action of the Galois group on the vertices known as the $\ast$-action.
\end{enumerate}
An algebraic $k$-group is called \textit{quasi-split} if its semisimple anisotropic kernel is trivial. Two algebraic groups are called \textit{inner forms} of each other if they share the same $\ast$-action. Each class of inner forms contains exactly one quasi-split group.

We give here the classification of the absolutely almost simple, rank one p-adic groups, due to Tits \cite{Tits1}. Up to restriction of scalars and passing to a different p-adic field, this yields all almost simple groups (including those that are not absolutely almost simple). We note that the description we give is not always that of the simply connected group - when the Dynkin diagram is of type $D_n$, one needs to pass to the simply connected cover. We use Tits' notation $^gX^t_{n,r}$, as in \cite{Tits2}, for the Tits index. Here, \textit{X} is the type of the Dynkin diagram (i.e, the group over the algebraic closure); \textit{n} and \textit{r} are respectively the rank over the algebraic closure and the rank over $k$; \textit{g} denotes the order of the quotient of the Galois group that acts faithfully on the diagram; \textit{t} stands for the degree of a division algebra occurring in the definition (in the exceptional cases it stands for something else, but this is irrelevant for rank one p-adic groups). The Tits index of an anisotropic kernel may be deduced from the Tits index of a group by erasing all distinguished orbits. Absolutely simple anisotropic p-adic groups are always inner forms of $A_n$.
\begin{enumerate}
    \item The group $SL_2(k)$. It is the split form of $SL_2$. Its Tits index notation is $^1A^1_{1,1}$, and its Tits index diagram is:
    \[
    \begin{tikzpicture}
    \draw (0,0) node (a1) {$\bullet$};
    \draw (a1) circle[radius=2mm];
    \end{tikzpicture}
    \]
    \item The group $SL_2(D)$, where $D$ is a central division algebra of degree $d \geq 2$ over $k$. It is a non-quasi-split form of $SL_{2d}$. Its Tits index notation is $^1A^d_{2d-1,1}$, and its Tits index diagram is:
    \[
    \begin{tikzpicture}
    \path[-] (0,0) node (a1) {$\bullet$}
             (1, 0) node (dots1) {$\cdots$}
             (2, 0) node (a2) {$\bullet$}
             (3, 0) node (a3) {$\bullet$}  
             (4, 0) node (a4) {$\bullet$} 
             (5, 0) node (dots2) {$\cdots$}
             (6, 0) node (a5) {$\bullet$};
    \draw[-] (a1) -- (dots1) -- (a2) -- (a3) -- (a4) -- (dots2) -- (a5);
    \draw (a3) circle[radius=2mm];
    \end{tikzpicture}
    \]
    \item The group $SU_3(h)$, where $k' / k$ is a quadratic extension, with an involution $\sigma$ of the second kind and $h$ a non-degenerate hermitian form relative to $\sigma$. It is a quasi-split form of $SL_3$. Its Tits index notation is $^2A^1_{2,1}$, and its Tits index diagram is:
    
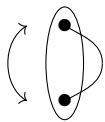
\[
    \begin{tikzpicture}
    \path (0,0) node {$\bullet$}
          (0,0.5) node (a) {}
          (0,1) node {$\bullet$};
    \draw (a) ellipse[x radius=2.5mm, y radius=7.5mm];
    \draw [-] (0,0) to [out=30,in=-90] (0.5,0.5);
    \draw [-] (0,1) to [out=-30,in=90] (0.5,0.5);
    \draw [<->] (-0.5,1) to [out=-150,in=150] (-0.5,0);
    \end{tikzpicture}
    \]
    \item The group $SU_4(h)$, where $k' / k$ is a quadratic extension, with an involution $\sigma$ of the second kind and $h$ a non-degenerate hermitian form relative to $\sigma$. It is a non-quasi-split form of $SL_4$. Its Tits index notation is $^2A^1_{3,1}$. The same index corresponds to the group $SU_3(D,h)$, where $D$ is a central division algebra of degree $2$ over $k$, with an involution $\sigma$ of the first kind, first type and $h$ is a non-degenerate hermitian form relative to $\sigma$. In the second case, Tits chooses to denote the index by $^2D^2_{3,1}$. The Tits index diagram is:
    
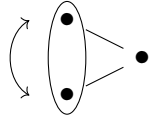
\[
    \begin{tikzpicture}
    \path (0,0) node (a1) {$\bullet$}
          (0,0.5) node (c) {}
          (0,1) node (a3) {$\bullet$}
          (1,0.5) node (a2) {$\bullet$};
    \draw[-] (a1) -- (a2) -- (a3);
    \draw (c) ellipse[x radius=2.5mm,y radius=7.5mm];
    \draw [<->] (-0.5,1) to [out=-150,in=150] (-0.5,0);
    \end{tikzpicture}
    \]
    \item The group $SU_2(D,h)$ where $D$ is a quaternion division algebra over $k$, with an involution $\sigma$ of the first kind, first type and $h$ is a non-degenerate skew-hermitian form relative to $\sigma$. It is a non-quasi-split form of $Sp_4$. Its Tits index notation is $C^2_{2,1}$, and its Tits index diagram is:
    \[
    \begin{tikzpicture}
    \path (0,0) node (a1) {$\bullet$}
          (0.4,0) node (arrow) {$\impliedby$}
          (1,0) node (a2) {$\bullet$};
    \draw (a2) circle[radius=2mm];
    \end{tikzpicture}
    \]
    \item The group $SU_3(D,h)$ where $D$ is a quaternion division algebra over $k$, with an involution $\sigma$ of the first kind, first type and $h$ is a non-degenerate skew-hermitian form relative to $\sigma$. It is a non-quasi-split form of $Sp_6$. Its Tits index notation is $C^2_{3,1}$, and its Tits index diagram is:
    \[
    \begin{tikzpicture}
       \path (0,0) node (a1) {$\bullet$}
             (1,0) node (a2) {$\bullet$}
             (2,0) node (a3) {$\bullet$}
             (1.5,0) node (arrow) {$\impliedby$};
    \draw[-] (a1) -- (a2);         
    \draw (a2) circle[radius=2mm];
    \end{tikzpicture}
    \]
    \item The group $SU_4(D,h)$ where $D$ is a central division algebra over $k$ of degree $2$, with an involution $\sigma$ of the first kind, first type and $h$ is a non-degenerate hermitian form relative to $\sigma$. It is a non-quasi-split form of $SO_8$. Its Tits index notation is $^2D^2_{4,1}$, and its Tits index diagram is:
    
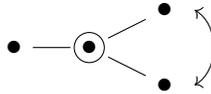
\[
    \begin{tikzpicture}
    \path (0,0) node (a1) {$\bullet$}
          (1,0) node (a2) {$\bullet$}
          (2,0.5) node (a3) {$\bullet$}
          (2,-0.5) node (a4) {$\bullet$};
    \draw [-] (a1) -- (a2) -- (a3);
    \draw [-] (a2) -- (a4);
    \draw (a2) circle[radius=2mm];
    \draw [<->] (2.4, 0.5) to [out=-30,in=30] (2.4,-0.5);
    \end{tikzpicture}
    \]
    \item The group $SU_5(D,h)$ where $D$ is a central division algebra over $k$ of degree $2$, with an involution $\sigma$ of the first kind, first type and $h$ is a non-degenerate hermitian form relative to $\sigma$. It is a non-quasi-split form of $SO_{10}$. Its Tits index notation is $^1D^2_{5,1}$, and its Tits index diagram is:
    \[
    \begin{tikzpicture}
    \path (0,0) node (a1) {$\bullet$}
          (1,0) node (a2) {$\bullet$}
          (2,0) node (a3) {$\bullet$}
          (3,0.5) node (a4) {$\bullet$}
          (3,-0.5) node (a5) {$\bullet$};
    \draw [-] (a1) -- (a2) -- (a3) -- (a4);
    \draw [-] (a3) -- (a5);
    \draw (a2) circle[radius=2mm];
    \end{tikzpicture}
    \]
\end{enumerate}

\bibliographystyle{amsplain}
\bibliography{bibliography}
\end{document}